\documentclass[12pt]{amsart}
\newtheorem{theorem}[subsection]{Theorem}
\newtheorem{proposition}[subsection]{Proposition}
\newtheorem{corollary}[subsection]{Corollary}
\numberwithin{equation}{section}
\newcommand{\ffrac}[2]{{\mbox{\large$\frac{#1}{#2}$}}}
\newcommand{\Rho}{\mathrm{P}}
\newcommand{\intprod}{\;\rule{5pt}{.3pt}\rule{.3pt}{7pt}\;}
\begin{document}
\title{Metric Connections in Projective Differential Geometry}
\author{Michael Eastwood}
\address{Department of Mathematics, University of Adelaide,
SA 5005, Australia}
\email{meastwoo@member.ams.org}
\author{Vladimir S.  Matveev}
\address{Mathematisches Institut,
Fakult\"at f\"ur Mathematik und Informatik,\newline
Friedrich-Schiller-Universit\"at Jena,
07737 Jena, Germany}
\email{matveev@minet.uni-jena.de}
\thanks{This work was undertaken during the 2006 Summer Program at the
Institute for Mathematics and its Applications at the University of Minnesota.
The authors would like to thank the IMA for hospitality during this time. 
The first author is supported by the Australian Research Council.}
\begin{abstract}
We search for Riemannian metrics whose Levi-Civita connection belongs to a
given projective class. Following Sinjukov and Mike\v{s}, we show that such
metrics correspond precisely to suitably positive solutions of a certain
projectively invariant finite-type linear system of partial differential
equations. Prolonging this system, we may reformulate these equations as
defining covariant constant sections of a certain vector bundle with
connection. This vector bundle and its connection are derived from the Cartan
connection of the underlying projective structure.
\end{abstract}
\dedicatory{In memory of Thomas Branson} 
\renewcommand{\subjclassname}{\textup{2000} Mathematics Subject Classification}
\maketitle

\section{Introduction}
We shall always work on a smooth oriented manifold $M$ of dimension~$n$.
Suppose that $\nabla$ is a torsion-free connection on the tangent bundle
of~$M$. We may ask whether there is a Riemannian metric on $M$ whose geodesics
coincide with the geodesics of $\nabla$ as unparameterised curves. We shall
show that there is a linear system of partial differential equations that
precisely controls this question.

To state our results, we shall need some terminology, notation, and preliminary
observations. Two torsion-free connections $\nabla$ and $\hat\nabla$ are said
to be projectively equivalent if they have the same geodesics as
unparameterised curves. A projective structure on $M$ is a projective
equivalence class of connections. In these terms, we are given a projective
structure on $M$ and we ask whether it may be represented by a metric
connection. Questions such as this have been addressed by many authors.
Starting with a metric connection, Sinjukov \cite{s} considered the existence
of other metrics with the same geodesics. He found a system of equations that
controls this question and Mike\v{s} \cite{m} observed that essentially the
same system pertains when starting with an arbitrary projective structure.

We shall use Penrose's abstract index notation~\cite{ot} in which indices act 
as markers to specify the type of a tensor. Thus, $\omega_a$ denotes a $1$-form
whilst $X^a$ denotes a vector field. Repeated indices denote the canonical 
pairing between vectors and co-vectors. Thus, we shall write $X^a\omega_a$ 
instead of~$X\intprod\omega$. The tautological $1$-form with values in the 
tangent bundle is denoted by the Kronecker delta~$\delta_a{}^b$. 

As is well-known~\cite{e}, the geometric formulation of projective equivalence
may be re-expressed as
\begin{equation}\label{projectivechange}
\hat\nabla_aX^b=\nabla_aX^b+\Upsilon_aX^b+\delta_a{}^b\Upsilon_cX^c
\end{equation}
for an arbitrary $1$-form~$\Upsilon_a$. We shall also adopt the curvature 
conventions of~\cite{e}. In particular, it is convenient to write  
$$(\nabla_b\nabla_a-\nabla_a\nabla_b)X^b=R_{ab}X^b,$$
where $R_{ab}$ is the usual Ricci tensor, as 
$$(\nabla_b\nabla_a-\nabla_a\nabla_b)X^b=(n-1)\Rho_{ab}X^b-\beta_{ab}X^b\quad
\mbox{where }\beta_{ab}=\Rho_{ba}-\Rho_{ab}.$$
If a different connection is chosen in the projective class according
to~(\ref{projectivechange}), then 
$$\hat\beta_{ab}=\beta_{ab}+\nabla_a\Upsilon_b-\nabla_b\Upsilon_a.$$
Therefore, as a $2$-form $\beta_{ab}$ changes by an exact form. On the other
hand, the Bianchi identity implies that $\beta_{ab}$ is closed. Thus, there is
a well-defined de~Rham cohomology class $[\beta]\in H^2(M,{\mathbb{R}})$
associated to any projective structure. 
\begin{proposition} The class $[\beta]\in H^2(M,{\mathbb{R}})$ is an
obstruction to the existence of a metric connection in the given projective
class.
\end{proposition}
\begin{proof}
The Ricci tensor is symmetric for a metric connection.
\end{proof}
In searching for a metric connection in a given projective class, we may as
well suppose that the obstruction $[\beta]$ vanishes. For the remainder of this 
article we suppose that this is the case and we shall consider only 
representative connections with symmetric Ricci tensor. In other words, all 
connections from now on enjoy
\begin{equation}\label{schouten}
(\nabla_b\nabla_a-\nabla_a\nabla_b)X^b=(n-1)\Rho_{ab}X^b\quad\mbox{where }
\Rho_{ab}=\Rho_{ba}.\end{equation}
A convenient alternative characterisation of such connections as follows.
\begin{proposition}
A torsion-free affine connection has symmetric Ricci tensor if and only if it 
induces the flat connection on the bundle of $n$-forms.\end{proposition}
\begin{proof}If $\epsilon^{pqr\cdots s}$ has $n$ indices and is totally skew 
then 
$$(\nabla_a\nabla_b-\nabla_b\nabla_a)\epsilon^{pqr\cdots s}=
\kappa_{ab}\epsilon^{pqr\cdots s}$$
for some $2$-form $\kappa_{ab}$. But, by the Bianchi symmetry,
$$(\nabla_a\nabla_b-\nabla_b\nabla_a)\epsilon^{abr\cdots s}=
-2R_{ab}\epsilon^{abr\cdots s},$$
which vanishes if and only if $R_{ab}$ is symmetric. 
\end{proof} 
Having restricted our attention to affine connections that are flat on the
bundle of $n$-forms, we may as well further restrict to connections $\nabla_a$
for which there is a volume form~$\epsilon_{bc\cdots d}$ (unique up to scale)
with $\nabla_a\epsilon_{bc\cdots d}=0$. We shall refer to such connections as
special. The freedom in special connections within a given projective class is
given by (\ref{projectivechange}) where $\Upsilon_a=\nabla_af$ for an arbitrary
smooth function~$f$. Following~\cite{e}, the full curvature of a special 
connection may be conveniently decomposed:--
\begin{equation}\label{fullcurvature}
(\nabla_a\nabla_b-\nabla_b\nabla_a)X^c=
W_{ab}{}^c{}_dX^d+\delta_a{}^c\Rho_{bd}X^d-\delta_b{}^c\Rho_{ad}X^d,
\end{equation}
where $W_{ab}{}^c{}_d$ is totally trace-free and $\Rho_{ab}$ is symmetric. The 
tensor $W_{ab}{}^c{}_d$ is known as the Weyl curvature and is projectively 
invariant. 

\section{A linear system of equations}\label{linearsystem}
In this section we present, as Proposition~\ref{alternativeLeviCivita}, an
alternative characterisation of the Levi-Civita connection. The advantage of
this characterisation is that it leads, almost immediately, to a system of
linear equations that controls the metric connections within a given projective
class. The precise results are Theorems~\ref{basictheorem} and~\ref{control}.

\begin{proposition}\label{alternativeLeviCivita} 
Suppose $g^{ab}$ is a metric on 
$M$ with volume form $\epsilon_{bc\cdots d}$. Then a torsion-free connection 
$\nabla_a$ is the metric connection for $g^{ab}$ if and only if
\begin{itemize}
\item $\nabla_ag^{bc}=\delta_a{}^b\mu^c+\delta_a{}^c\mu^b$\quad for some vector
field $\mu^a$
\item $\nabla_a\epsilon_{bc\cdots d}=0$.
\end{itemize}
\end{proposition}
\begin{proof}
Write $D_a$ for the metric connection of~$g^{ab}$. Then
\begin{equation}\label{contorsion}
\nabla_a\omega_b=D_a\omega_b-\Gamma_{ab}{}^c\omega_c
\end{equation}
for some tensor $\Gamma_{ab}{}^c=\Gamma_{ba}{}^c$. We compute
$$\epsilon^{bc\cdots d}\nabla_a\epsilon_{bc\cdots d}=
-n\epsilon^{bc\cdots d}\Gamma_{ab}{}^e\epsilon_{ec\cdots d}=
-n!\,\Gamma_{ab}{}^b$$
and so $\Gamma_{ab}{}^b=0$. Similarly,
$$\nabla_ag^{bc}=\Gamma_{ad}{}^bg^{dc}+\Gamma_{ad}{}^cg^{bd}$$
and so
\begin{equation}\label{andso}
\Gamma_{ad}{}^bg^{dc}+\Gamma_{ad}{}^cg^{bd}=
\delta_a{}^b\mu^c+\delta_a{}^c\mu^b.
\end{equation}
Let $g_{ab}$ denote the inverse of $g^{ab}$ and contract (\ref{andso}) with 
$g_{bc}$ to conclude that
$$2\Gamma_{ab}{}^b=2\mu_a\quad\mbox{where }\mu_a=g_{ab}\mu^b$$
and hence that $\mu^a=0$. If we let $\Gamma_{abc}=\Gamma_{ab}{}^dg_{cd}$, then 
(\ref{andso}) now reads
$$\Gamma_{acb}+\Gamma_{abc}=0.$$
Together with $\Gamma_{abc}=\Gamma_{bac}$, this implies that
$\Gamma_{abc}=0$. {From} (\ref{contorsion}) we see that $\nabla_a=D_a$,
which is what we wanted to show.
\end{proof}

\begin{theorem}\label{basictheorem}
Suppose $\nabla_a$ is a special torsion-free connection and there is a metric 
tensor $\sigma^{ab}$ such that 
\begin{equation}\label{key}
\nabla_a\sigma^{bc}=\delta_a{}^b\mu^c+\delta_a{}^c\mu^b\quad\mbox{for some 
vector field }\mu^a.
\end{equation}
Then $\nabla_a$ is projectively equivalent to a metric connection.
\end{theorem}
\begin{proof}
Consider the projectively equivalent connection    
$$\hat\nabla_aX^b=\nabla_aX^b+\Upsilon_aX^b+\delta_a{}^b\Upsilon_cX^c\quad
\mbox{where }\Upsilon_a=\nabla_af$$
for some function~$f$. If we let $\hat\sigma^{ab}\equiv e^{-2f}\sigma^{ab}$, 
then 
$$\begin{array}{rcl}
\hat\nabla_a\hat\sigma^{bc}&=&e^{-2f}
\left(-2\Upsilon_a\sigma^{bc}+\nabla_a\sigma^{bc}+2\Upsilon_a\sigma^{bc}
+\delta_a{}^b\Upsilon_d\sigma^{dc}+\delta_a{}^c\Upsilon_d\sigma^{bd}
\right)\\[3pt]
&=&e^{-2f}
\left(\delta_a{}^b\mu^c+\delta_a{}^c\mu^b
+\delta_a{}^b\Upsilon_d\sigma^{dc}+\delta_a{}^c\Upsilon_d\sigma^{bd}\right)
\end{array}$$
and so
\begin{equation}\label{hatsigmakill}
\hat\nabla_a\hat\sigma^{bc}=\delta_a{}^b\hat\mu^c+\delta_a{}^c\hat\mu^b\quad
\mbox{where }\hat\mu^a=e^{-2f}\left(\mu^a+\Upsilon_b\sigma^{ab}\right).
\end{equation}
Similarly, if we choose a volume form $\epsilon_{bc\cdots d}$ killed by
$\nabla_a$ and let $\hat\epsilon_{bc\cdots d}\equiv 
e^{(n+1)f}\epsilon_{bc\cdots d}$, then
\begin{equation}\label{hatepsilonkill}
\hat\nabla_a\hat\epsilon_{bc\cdots d}=
e^{(n+1)f}\left(\nabla_a\epsilon_{bc\cdots d}+
\Upsilon_{[a}\epsilon_{bc\cdots d]}\right)
=e^{(n+1)f}\nabla_a\epsilon_{bc\cdots d}=0.
\end{equation}
Define 
$$\det(\sigma)\equiv
\epsilon_{a\cdots b}\epsilon_{c\cdots d}\sigma^{ac}\cdots\sigma^{bd}$$
and compute
$$\begin{array}{rcl}\widehat{\det}(\hat\sigma)&=&
\hat\epsilon_{a\cdots b}\hat\epsilon_{c\cdots d}
\hat\sigma^{ac}\cdots\hat\sigma^{bd}\\[3pt]
&=&e^{2(n+1)f}e^{-2nf}
\epsilon_{a\cdots b}\epsilon_{c\cdots d}\sigma^{ac}\cdots\sigma^{bd}
\,=\,e^{2f}\det(\sigma).
\end{array}$$
Therefore, if we take 
$$f=-\ffrac{1}{2}\log\det(\sigma),$$
then we have arranged that $\widehat{\det}(\hat\sigma)=1$. This is precisely
the condition that $\hat\epsilon_{bc\cdots d}$ be the volume form for the
metric $\hat\sigma^{ab}$. With (\ref{hatsigmakill}) and (\ref{hatepsilonkill})
we are now in a position to use Proposition~\ref{alternativeLeviCivita} to
conclude that $\hat\nabla_a$ is the metric connection for $\hat\sigma^{ab}$. We
have shown that our original connection $\nabla_a$ is projectively equivalent
to the Levi-Civita connection for the metric
$g^{ab}\equiv\det(\sigma)\,\sigma^{ab}$.
\end{proof}
Evidently, the equations (\ref{key}) precisely control the metric
connections within a given special projective class. Precisely, if $g_{ab}$ is
a Riemannian metric with associated Levi-Civita connection~$\nabla_a$, then
$$\hat\nabla_a\hat g^{bc}=\delta_a{}^b\hat\mu^c+\delta_a{}^c\hat\mu^b,$$
where $\hat\nabla_a$ is projectively equivalent to $\nabla_a$ according to 
(\ref{projectivechange}) with $\Upsilon_a=\nabla_af$ and where 
$\hat g^{bc}=e^{-2f}g^{bc}$. In other words, we have shown (cf.~\cite{m,s}):--
\begin{theorem}\label{control}
There is a one-to-one correspondence between solutions
of~{\,\rm(\ref{key})} for positive definite~$\sigma^{bc}$ and metric 
connections that are projectively equivalent to~$\nabla_a$.
\end{theorem}

\section{Prolongation}
Let us consider the system of equations (\ref{key}) in more detail. It is 
a linear system for any symmetric contravariant $2$-tensor~$\sigma^{bc}$. 
Specifically, we may write (\ref{key}) as 
\begin{equation}\label{tfp}
\mbox{the trace-free part of }(\nabla_a\sigma^{bc})=0\end{equation}
or, more explicitly, as
$$\nabla_a\sigma^{bc}-\ffrac{1}{n+1}\delta_a{}^b\nabla_d\sigma^{cd}
-\ffrac{1}{n+1}\delta_a{}^c\nabla_d\sigma^{bd}=0.$$
According to~\cite{bceg}, this equation is of finite-type and may be prolonged 
to a closed system as follows.
According to (\ref{schouten}) and (\ref{fullcurvature}) we have
$$(\nabla_a\nabla_b-\nabla_b\nabla_a)\sigma^{bc}=W_{ab}{}^c{}_d\sigma^{bd}
+\delta_a{}^c\Rho_{bd}\sigma^{bd}-n\Rho_{ad}\sigma^{cd}.$$
On the other hand, from (\ref{key}) we have
$$(\nabla_a\nabla_b-\nabla_b\nabla_a)\sigma^{bc}=
(n+1)\nabla_a\mu^c-\nabla_b(\delta_a{}^b\mu^c+\delta_a{}^c\mu^b)=
n\nabla_a\mu^c-\delta_a{}^c\nabla_b\mu^b.$$
We conclude that 
$$n\nabla_a\mu^c=
\delta_a{}^c\left(\nabla_b\mu^b+\Rho_{bd}\sigma^{bd}\right)
-n\Rho_{ad}\sigma^{cd}+W_{ab}{}^c{}_d\sigma^{bd}$$
or, equivalently, that
\begin{equation}\label{second}\nabla_a\mu^c=\delta_a{}^c\rho
-\Rho_{ad}\sigma^{cd}+\ffrac{1}{n}W_{ab}{}^c{}_d\sigma^{bd},\end{equation}
for some function $\rho$. To complete the prolongation, we use 
(\ref{schouten}) to write
$$(\nabla_c\nabla_a-\nabla_a\nabla_c)\mu^c=(n-1)\Rho_{ac}\mu^c$$
whereas from (\ref{second}) we also have
$$(\nabla_c\nabla_a-\nabla_a\nabla_c)\mu^c=\nabla_c\left(
\delta_a{}^c\rho
-\Rho_{ad}\sigma^{cd}+\ffrac{1}{n}W_{ab}{}^c{}_d\sigma^{bd}\right)
-\nabla_a\left(n\rho
-\Rho_{cd}\sigma^{cd}\right).$$
Therefore,
\begin{equation}\label{hold}(n-1)\Rho_{ac}\mu^c=\nabla_c\left(
-\Rho_{ad}\sigma^{cd}+\ffrac{1}{n}W_{ab}{}^c{}_d\sigma^{bd}\right)
-(n-1)\nabla_a\rho
+\nabla_a(\Rho_{cd}\sigma^{cd}).\end{equation}
The terms involving Weyl curvature
$$\nabla_c(W_{ab}{}^c{}_d\sigma^{bd})=(\nabla_cW_{ab}{}^c{}_d)\sigma^{bd}
+W_{ab}{}^c{}_d\nabla_c\sigma^{bd}$$
may be dealt with by (\ref{key}) and a Bianchi identity
$$\nabla_cW_{ab}{}^c{}_d=(n-2)(\nabla_a\Rho_{bd}-\nabla_b\Rho_{ad}).$$
We see that
$$\nabla_c(W_{ab}{}^c{}_d\sigma^{bd})=
(n-2)(\nabla_a\Rho_{bd}-\nabla_b\Rho_{ad})\sigma^{bd}$$
and (\ref{hold}) becomes
$$(n-1)\Rho_{ac}\mu^c=
\ffrac{n-2}{n}(\nabla_a\Rho_{bd}-\nabla_b\Rho_{ad})\sigma^{bd}
-\nabla_c(
\Rho_{ad}\sigma^{cd})
-(n-1)\nabla_a\rho
+\nabla_a(\Rho_{cd}\sigma^{cd})$$
or, equivalently,
\begin{equation}\label{newhold}\Rho_{ac}\mu^c=
\ffrac{2}{n}(\nabla_a\Rho_{bd}-\nabla_b\Rho_{ad})\sigma^{bd}
-\ffrac{1}{n-1}\Rho_{ad}\nabla_c\sigma^{cd}
-\nabla_a\rho
+\ffrac{1}{n-1}\Rho_{cd}\nabla_a\sigma^{cd}.\end{equation}
Again, we substitute from (\ref{key}) to rewrite
$$\Rho_{cd}\nabla_a\sigma^{cd}-\Rho_{ad}\nabla_c\sigma^{cd}=
\Rho_{cd}(\delta_a{}^c\mu^d+\delta_a{}^d\mu^c)-(n+1)\Rho_{ad}\mu^d=
-(n-1)\Rho_{ad}\mu^d$$
and (\ref{newhold}) becomes
$$\Rho_{ac}\mu^c=
\ffrac{2}{n}(\nabla_a\Rho_{bd}-\nabla_b\Rho_{ad})\sigma^{bd}
-\Rho_{ad}\mu^d-\nabla_a\rho,$$
which we may rearrange as
$$\nabla_a\rho=-2\Rho_{ab}\mu^b+
\ffrac{2}{n}(\nabla_a\Rho_{bd}-\nabla_b\Rho_{ad})\sigma^{bd}.$$
Together with (\ref{key}) and~(\ref{second}), we have a closed system,
essentially as in~\cite{m,s}:--
\begin{equation}\label{closure}\begin{array}{rcl}
\nabla_a\sigma^{bc}&=&\delta_a{}^b\mu^c+\delta_a{}^c\mu^b\\[3pt]
\nabla_a\mu^b&=&\delta_a{}^b\rho
-\Rho_{ac}\sigma^{bc}+\ffrac{1}{n}W_{ac}{}^b{}_d\sigma^{cd}\\[3pt]
\nabla_a\rho&=&-2\Rho_{ab}\mu^b+
\ffrac{4}{n}Y_{abc}\sigma^{bc}
\end{array}\end{equation}
where $Y_{abc}=\ffrac{1}{2}(\nabla_a\Rho_{bc}-\nabla_b\Rho_{ac})$, the
Cotton-York tensor. The three tensors $\sigma^{bc}$, $\mu^b$, and $\rho$ may be
regarded together as a section of the vector bundle
$$\textstyle{\mathcal{T}}=\bigodot^2\!TM\oplus TM\oplus{\mathbb{R}}$$
where $\bigodot$ denotes symmetric tensor product and ${\mathbb{R}}$ denotes
the trivial bundle. We have proved:-- 
\begin{theorem}\label{prolongedconnection}
If we endow ${\mathcal{T}}$ with the connection
\begin{equation}\label{tampered}\left\lgroup\begin{array}c
\sigma^{bc}\\[3pt]
\mu^b\\[3pt]
\rho
\end{array}\right\rgroup\longmapsto
\left\lgroup\begin{array}c
\nabla_a\sigma^{bc}-\delta_a{}^b\mu^c-\delta_a{}^c\mu^b\\[3pt]
\nabla_a\mu^b-\delta_a{}^b\rho+\Rho_{ac}\sigma^{bc}
-\frac{1}{n}W_{ac}{}^b{}_d\sigma^{cd}\\[3pt]
\nabla_a\rho+2\Rho_{ab}\mu^b-\frac{4}{n}Y_{abc}\sigma^{bc}
\end{array}\right\rgroup.\end{equation}
then there is a one-to-one correspondence between covariant constant
sections of ${\mathcal{T}}$ and solutions $\sigma^{bc}$ of~{\,\rm(\ref{key})}.
\end{theorem}    

\section{Projective invariance}
The equation (\ref{key}) is projectively invariant in the following sense.
Following~\cite{e}, let ${\mathcal{E}}^{(ab)}(w)$ denote the bundle of
symmetric contravariant $2$-tensors of projective weight~$w$. Thus, in the
presence of a volume form $\epsilon_{bc\cdots d}$, a section
$\sigma^{ab}\in\Gamma(M,{\mathcal{E}}^{(ab)}(w))$ is an ordinary symmetric 
contravariant $2$-tensor but if we 
change volume form 
$$\epsilon_{bc\cdots d}\mapsto
\hat\epsilon_{bc\cdots d}=e^{(n+1)f}\epsilon_{bc\cdots d}\quad
\mbox{for any smooth function }f,$$
then we are obliged to rescale $\sigma^{ab}$ according to
$\hat\sigma^{ab}=e^{wf}\sigma^{ab}$. Equivalently, we are saying that
${\mathcal{E}}^{(ab)}(w)=\bigodot^2\!TM\otimes(\Lambda^n)^{-w/(n+1)}$, where
$\Lambda^n$ is the line-bundle of $n$-forms on~$M$. The projectively weighted
irreducible tensor bundles are fundamental objects on a manifold with
projective structure.
\begin{proposition}
The differential operator
\begin{equation}\label{tfpop}{\mathcal{E}}^{(ab)}(-2)\longrightarrow
\mbox{\rm the trace-free part of }{\mathcal{E}}_a{}^{(bc)}(-2)\end{equation}
defined by {\rm(\ref{tfp})} is projectively invariant.
\end{proposition}
\begin{proof}
This is already implicit in the proof of Theorem~\ref{basictheorem}.
Explicitly, however, we just compute from (\ref{projectivechange}):--
$$\hat\nabla_a\hat\sigma^{bc}=
\nabla_a\hat\sigma^{bc}+2\Upsilon_a\hat\sigma^{bc}
+\delta_a{}^b\Upsilon_d\hat\sigma^{dc}
+\delta_a{}^c\Upsilon_d\hat\sigma^{bd},$$
where $\Upsilon_a=\nabla_af$ whilst
$$\nabla_a\hat\sigma^{bc}=\nabla_a(e^{-2f}\sigma^{bc})=
e^{-2f}\left(\nabla_a\sigma^{bc}-2\Upsilon_a\sigma^{bc}\right)=
\widehat{\nabla_a\sigma^{bc}}-2\Upsilon_a\hat\sigma^{bc}.$$
It follows that
$$\hat\nabla_a\hat\sigma^{bc}=\widehat{\nabla_a\sigma^{bc}}+\mbox{trace 
terms},$$
which is what we wanted to show.
\end{proof}

In hindsight, it is not too difficult to believe that (\ref{tfp}) should
control the metric connections within a given projective class. There are very
few projectively invariant operators. In fact, there are precisely two
finite-type first order invariant linear operators on symmetric $2$-tensors.
One of them is (\ref{tfpop}) and the other is
\begin{equation}\label{theotherone}
{\mathcal{E}}_{(ab)}(4)\rightarrow{\mathcal{E}}_{(abc)}(4)\quad
\mbox{given by}\quad\sigma_{ab}\mapsto\nabla_{(a}\sigma_{bc)}.\end{equation}
In two dimensions, (\ref{theotherone}) and (\ref{tfpop}) coincide. In higher
dimensions, however, being in the kernel of (\ref{theotherone}) for positive
definite $\sigma_{ab}$ corresponds to having a metric $g_{ab}$ and a totally
trace-free tensor $\Gamma_{abc}$ with 
$$\Gamma_{abc}=\Gamma_{bac}
\quad\mbox{and}\quad
\Gamma_{abc}+\Gamma_{bca}+\Gamma_{cab}=0$$
such that the connection
$$\omega_b\longmapsto D_a\omega_b-\Gamma_{ab}{}^c\omega_c$$
belongs to the projective class of~$\nabla_a$, where $D_a$ is the Levi-Civita
connection of~$g_{ab}$. The available tensors $\Gamma_{abc}$ for a given metric
have dimension $n(n+2)(n-2)/3$.

\section{Relationship to the Cartan connection}
On a manifold with projective structure, it is shown in \cite{e} how to
associate vector bundles with connection to any irreducible representation of
${\mathrm{SL}}(n+1,{\mathbb{R}})$. These are the tractor bundles following
their construction by Thomas~\cite{t}. Equivalently, they are induced by the
Cartan connection \cite{c} of the projective structure.
The re\-le\-vant tractor bundle in our case is induced by
$\bigodot^2\!{\mathbb{R}}^{n+1}$ where ${\mathbb{R}}^{n+1}$ is the defining
representation of ${\mathrm{SL}}(n+1,{\mathbb{R}})$. It has a composition 
series
$${\mathcal{E}}^{(BC)}
={\mathcal{E}}^{(bc)}(-2)+{\mathcal{E}}^b(-2)+{\mathcal{E}}(-2)$$
and in the presence of a connection is simply the direct sum of these bundles.
Under projective change of connection according to (\ref{projectivechange}),
however, we decree that
\begin{equation}\label{changetractors}\widehat{\left\lgroup\begin{array}c
\sigma^{bc}\\[3pt]
\mu^b\\[3pt]
\rho
\end{array}\right\rgroup}=\left\lgroup\begin{array}c
\sigma^{bc}\\[3pt]
\mu^b+\Upsilon_c\sigma^{bc}\\[3pt]
\rho+2\Upsilon_b\mu^b+\Upsilon_b\Upsilon_c\sigma^{bc}
\end{array}\right\rgroup.\end{equation}
Following~\cite{e}, the tractor connection on
${\mathcal{E}}^{(AB)}$ is given by
$$\nabla_a\left\lgroup\begin{array}c
\sigma^{bc}\\[3pt]
\mu^b\\[3pt]
\rho
\end{array}\right\rgroup=
\left\lgroup\begin{array}c
\nabla_a\sigma^{bc}-\delta_a{}^b\mu^c-\delta_a{}^c\mu^b\\[3pt]
\nabla_a\mu^b-\delta_a{}^b\rho
+\Rho_{ac}\sigma^{bc}\\[3pt]
\nabla_a\rho+2\Rho_{ab}\mu^b
\end{array}\right\rgroup.$$
Therefore, we have proved:--
\begin{theorem}\label{modifiedtractors}
The solutions of {\,\rm(\ref{key})} are in one-to-one
correspondence with solutions of the following system:--
\begin{equation}\label{newkey}\nabla_a\left\lgroup\begin{array}c
\sigma^{bc}\\
\mu^b\\
\rho
\end{array}\right\rgroup-\frac{1}{n}\left\lgroup\begin{array}c
0\\
W_{ac}{}^b{}_d\sigma^{cd}\\
4Y_{abc}\sigma^{bc}\end{array}\right\rgroup=0.\end{equation}
\end{theorem}
\begin{corollary}
There is a one-to-one correspondence between solutions
of~{\,\rm(\ref{newkey})} for positive definite~$\sigma^{bc}$ and metric 
connections that are projectively equivalent to~$\nabla_a$.  
\end{corollary}
Notice that the extra terms in (\ref{newkey}) are projectively invariant as
they should be. Specifically, it is observed in \cite{e} that 
$$\textstyle\hat Y_{abc}=Y_{abc}+\frac{1}{2}W_{ab}{}^d{}_c\Upsilon_d$$
and so
$$4\hat Y_{abc}\sigma^{bc}=
4Y_{abc}\sigma^{bc}+2\Upsilon_bW_{ac}{}^b{}_d\sigma^{cd}$$
in accordance with~(\ref{changetractors}).

It is clear from Theorem~\ref{prolongedconnection} that, generically,
(\ref{key}) has no solutions. Indeed, this is one reason why the prolonged from
is so helpful. More generally, we should compute the curvature of the
connection (\ref{tampered}) and the form (\ref{newkey}) is useful for this
task. A model computation along these lines is given in~\cite{e}. In our case,
the tractor curvature is given by
$$(\nabla_a\nabla_b-\nabla_b\nabla_a)\left\lgroup\begin{array}c
\sigma^{cd}\\[3pt]
\mu^c\\[3pt]
\rho
\end{array}\right\rgroup=\left\lgroup\begin{array}c
W_{ab}{}^c{}_e\sigma^{de}+W_{ab}{}^d{}_e\sigma^{ce}\\[3pt]
W_{ab}{}^c{}_d\mu^d+2Y_{abd}\sigma^{cd}\\[3pt]
4Y_{abc}\mu^c
\end{array}\right\rgroup$$
and we obtain:--
\begin{proposition}
The curvature of the connection {\rm(\ref{tampered})} is given by
$$\left\lgroup\begin{array}c
\sigma^{cd}\\[3pt]
\mu^c\\[3pt]
\rho
\end{array}\right\rgroup\mapsto\left\lgroup\begin{array}c
W_{ab}{}^c{}_e\sigma^{de}+W_{ab}{}^d{}_e\sigma^{ce}\\[3pt]
W_{ab}{}^c{}_d\mu^d+2Y_{abd}\sigma^{cd}\\[3pt]
4Y_{abc}\mu^c
\end{array}\right\rgroup
+\frac{1}{n}\left\lgroup\begin{array}c
\delta_a{}^cU_b{}^d+\delta_a{}^dU_b{}^c-
\delta_b{}^cU_a{}^d-\delta_b{}^dU_a{}^c\\[3pt]
\ast\\[3pt]
\ast
\end{array}\right\rgroup,$$
where $U_b{}^d=W_{be}{}^d{}_f\sigma^{ef}$ and $\ast$ denotes expressions that
we shall not need.
\end{proposition}
\begin{corollary}\label{mobility}
The curvature of the connection {\rm(\ref{tampered})} vanishes if and only if
the projective structure is flat.
\end{corollary}
\begin{proof}
Let us suppose that $n\geq 3$. The uppermost entry of the curvature is given by
$$\sigma^{cd}\longmapsto\mbox{the trace-free part of }
(W_{ab}{}^c{}_e\sigma^{de}+W_{ab}{}^d{}_e\sigma^{ce})$$
and it is a matter of elementary representation theory to show that if this
expression is zero for a fixed $W_{ab}{}^c{}_d$ and for all~$\sigma^{cd}$, then
$W_{ab}{}^c{}_d=0$. Specifically, the symmetries of $W_{ab}{}^c{}_d$, namely
\begin{equation}\label{weylsymmetries}W_{ab}{}^c{}_d+W_{ba}{}^c{}_d=0\qquad
W_{ab}{}^c{}_d+W_{bd}{}^c{}_a+W_{da}{}^c{}_b=0\qquad W_{ab}{}^a{}_d=0,
\end{equation}
constitute an irreducible representation of ${\mathrm{SL}}(n,{\mathbb{R}})$.
Hence, the submodule
$$\left\{W_{ab}{}^c{}_d\;\mbox{ s.t. }\mbox{the trace-free part of }
(W_{ab}{}^c{}_e\sigma^{de}+W_{ab}{}^d{}_e\sigma^{ce})=0,
\;\forall\;\sigma^{cd}\right\}$$
must be zero since it is not the whole space. We have shown that if the
curvature of the connection (\ref{tampered}) vanishes, then $W_{ab}{}^c{}_d=0$.
For $n\geq 3$ this is exactly the condition that the projective structure be
flat. For $n=2$, the Weyl curvature $W_{ab}{}^c{}_d$ vanishes automatically
since the symmetries (\ref{weylsymmetries}) are too severe a constraint.
Instead, a similar calculation shows that $Y_{abc}=0$ and this is the condition
that the projective structure be flat.
\end{proof}

Following Mike\v{s}~\cite{m}, the dimension of the space of solutions of
(\ref{key}) is called the degree of mobility of the projective structure.
Theorem~\ref{prolongedconnection} implies that the degree of mobility is
bounded by $(n+1)(n+2)/2$ and Corollary~\ref{mobility} implies that this bound
is achieved only for the flat projective structure. Of course, the flat
projective structure may as well be represented by the flat connection
$\nabla_a=\partial/\partial x^a$ on~${\mathbb{R}}^n$, which is the Levi-Civita
connection for the standard Euclidean metric. In this case, we may use
(\ref{closure}) find the general solution of~(\ref{key}):--
\begin{equation}\label{generalsolution}
\sigma^{ab}=s^{ab}+x^am^b+x^bm^a+x^ax^br.
\end{equation} 
This form is positive definite near the origin if and only if $s^{ab}$ is
positive definite. We conclude that the general projectively flat metric near
the origin in ${\mathbb{R}}^n$ is
$$g^{ab}=\det(\sigma)\,\sigma^{ab},$$
where $\sigma^{ab}$ is as in (\ref{generalsolution}) for some positive definite
quadratic form~$s^{ab}$. In fact, these metrics are constant curvature. Rather
than prove this by calculation, there is an alternative as follows.
As already observed, the Weyl curvature $W_{ab}{}^c{}_d$ corresponds to an
irreducible representation of ${\mathrm{SL}}(n,{\mathbb{R}})$ characterised
by~(\ref{weylsymmetries}). In the presence of a metric~$g_{ab}$, however, we
should decompose $W_{ab}{}^c{}_d$ further under ${\mathrm{SO}}(n)$.
\begin{proposition} In the presence of a metric $g_{ab}$
\begin{equation}\label{weylsplit}
W_{ab}{}^c{}_d=C_{ab}{}^c{}_d
+\ffrac{1}{(n-1)(n-2)}\left(\delta_a{}^c\Phi_{bd}-\delta_b{}^c\Phi_{ad}\right)
+\ffrac{1}{n-2}\left(\Phi_a{}^cg_{bd}-\Phi_b{}^cg_{ad}\right)\end{equation}
where $C_{ab}{}^c{}_d$ is the Weyl part of the Riemann curvature tensor 
and $\Phi_{ab}$ is the trace-free part of the Ricci tensor. 
\end{proposition}
\begin{proof}According to~(\ref{fullcurvature}),
\begin{equation}\label{chalk}
R_{ab}{}^c{}_d=W_{ab}{}^c{}_d+\delta_a{}^c\Rho_{bd}-\delta_b{}^c\Rho_{ad}
\end{equation}
but the Riemann curvature decomposes according to 
\begin{equation}\label{cheese}
R_{abcd}=C_{abcd}+g_{ac}Q_{bd}-g_{bc}Q_{ad}+Q_{ac}g_{bd}-Q_{bc}g_{ad},
\end{equation}
where $Q_{ab}$ is the Schouten tensor
$$Q_{ab}=\ffrac{1}{n-2}\Phi_{ab}+\ffrac{1}{2n(n-1)}Rg_{ab}.$$
Comparing (\ref{chalk}) and (\ref{cheese}) leads, after a short computation,
to~(\ref{weylsplit}).
\end{proof}
\begin{corollary}
A projectively flat metric is constant curvature.
\end{corollary}    
\begin{proof} If $n\geq 3$ and the projective Weyl tensor vanishes then the
only remaining part of the Riemann curvature tensor is the scalar curvature. As
usual, a separate proof based on $Y_{abc}$ is needed for the case $n=2$.
\end{proof} 
This corollary is usually stated as follows. If a local diffeomorphism between
two Riemannian manifolds preserves geodesics and one of them is constant
curvature, then so is the other. This is a classical result due to
Beltrami~\cite{b}.

\section{Concluding remarks}
Results such as Theorem~\ref{control} and Theorem~\ref{modifiedtractors} are
quite common in projective, conformal, and other parabolic geometries. It is
shown in~\cite{e}, for example, that the Killing equation in Riemannian
geometry is projectively invariant and its solutions are in one-to-one
correspondence with covariant constant sections of the tractor bundle
${\mathcal{E}}_{[AB]}$ equipped with a connection that is derived from (but not
quite equal to) the tractor connection. The situation is completely parallel
for conformal Killing vectors in conformal geometry and, more generally, for
the infinitesimal automorphisms of parabolic geometries~\cite{andi}. It is
well-known that having an Einstein metric in a given conformal class is
equivalent to having a suitably positive covariant constant section of the
standard tractor bundle ${\mathcal{E}}^A$ equipped with its usual tractor
connection. Gover and Nurowski~\cite{gn} use this observation systematically to
find obstructions to the existence of an Einstein metric within a given
conformal class. We anticipate a similar use for Theorem~\ref{modifiedtractors}
in establishing obstructions to the existence of a metric connection within a
given projective class.


\begin{thebibliography}{XX}

\bibitem{b} E. Beltrami,
{\em Rizoluzione del problema: riportare i punti di una superficie sopra un 
piano in modo che le linee geodetiche vengano rappresentate da linee rette},
Ann. Mat. Pura Appl. {\bf 7} (1865) 185--204.

\bibitem{bceg} T.P.~Branson, A.~\v{C}ap, M.G.~Eastwood, and A.R.~Gover,
{\em Prolongations of geometric overdetermined systems},
Int. Jour. Math. {\bf 17} (2006) 641-664.

\bibitem{andi} A. \v{C}ap, 
{\em Infinitesimal automorphisms and deformations of parabolic geometries},
preprint ESI 1684 (2005), Erwin Schr\"odinger Institute, available at 
http://www.esi.ac.at.

\bibitem{c} E. Cartan,
{\em Sur les vari\'et\'es \`a connexion projective},
Bull. Soc. Math. France {\bf 52} (1924) 205--241.

\bibitem{e} M.G. Eastwood,
{\em Notes on projective differential geometry}, this volume.

\bibitem{gn} A.R. Gover and P. Nurowski,
{\em  Obstructions to conformally Einstein metrics in $n$ dimensions},
Jour. Geom. Phys. {\bf 56} (2006) 450--484.

\bibitem{m} J. Mike\v{s},
{\em Geodesic mappings of affine-connected and Riemannian spaces},
Jour. Math. Sci. {\bf 78} (1996) 311--333.

\bibitem{ot} R. Penrose and W. Rindler,
{\em Spinors and Space-time, vol. 1},
Cambridge University Press 1984.

\bibitem{s} N.S. Sinjukov, 
{\em Geodesic mappings of Riemannian spaces} (Russian), 
``Nauka,'' Moscow 1979.

\bibitem{t} T.Y. Thomas,
{\em Announcement of a projective theory of affinely connected manifolds}, 
Proc. Nat. Acad. Sci. {\bf 11} (1925) 588--589.

\end{thebibliography}
\end{document}